\documentclass[11pt]{article}
\usepackage{amsfonts}
\usepackage{amsmath}
\usepackage{amssymb,amsthm,amscd}
\usepackage{epic, eepic, latexsym}
\usepackage{graphics}
\usepackage{epsfig}
\usepackage{fullpage}
\newcommand{\tr}[0]{\operatorname{tr}}
\newcommand{\tphi}[0]{\tilde{\varphi}}

\newtheorem{theorem}{Theorem}[section]
\newtheorem{lemma}[theorem]{Lemma}

\title{Convergence of the parabolic complex Monge-Amp\`ere equation on compact Hermitian manifolds}
\author{Matt Gill}
\date{}

\begin{document}
\maketitle

\bigskip

\begin{abstract}
We prove $C^\infty$ convergence for suitably normalized solutions of the parabolic complex Monge-Amp\`ere equation on compact Hermitian manifolds. This provides a parabolic proof of a recent result of Tosatti and Weinkove. 
\end{abstract}

\bigskip

\section{Introduction}

Let $(M,g)$ be a compact Hermitian manifold of complex dimension $n$ and $\omega$ be the real $(1,1)$ form $\omega = \sqrt{-1} \sum_{i,j} g_{i\bar{j}} dz^i \wedge dz^{\bar{j}}.$ Let $F$ be a smooth function on $M$. We consider the parabolic complex Monge-Amp\`{e}re equation 
\begin{equation}\label{eq:ma} 
\frac{\partial \varphi}{\partial t} = \log{\frac{\det{(g_{i\bar{j}}+\partial_i \partial_{\bar{j}} \varphi)}}{\det{g_{i\bar{j}}}}} - F, \ \ g_{i\bar{j}} + \partial_i \partial_{\bar{j}} \varphi > 0
\end{equation}
with initial condition $\varphi(x,0) = 0$.

The study of this type of Monge-Amp\`{e}re equation originated in proving the Calabi conjecture. The proof of the conjecture reduced to assuming that $\omega$ is K\"{a}hler and finding a unique solution to the elliptic Monge-Amp\`{e}re equation 
\begin{equation}\label{eq:elliptic}
\log{\frac{\det{(g_{i\bar{j}}+\partial_i \partial_{\bar{j}} \varphi)}}{\det{g_{i\bar{j}}}}} = F, \ \ g_{i\bar{j}} + \partial_i \partial_{\bar{j}} \varphi > 0.
\end{equation} 
Calabi showed that if a solution to \eqref{eq:elliptic} exists, it is unique up to adding a constant to $\varphi$ \cite{Ca}. Yau used the continuity method to show that if $$\int_M e^F \omega^n = \int_M \omega^n$$ then \eqref{eq:elliptic} admits a smooth solution \cite{Yau}. The proof of Yau required \textit{a priori} $C^\infty$ estimates for $\varphi$.

Cao used Yau's estimates to show that in the K\"{a}hler case, \eqref{eq:ma} has a smooth solution for all time that converges to the unique solution of \eqref{eq:elliptic} \cite{Cao}. 

Since not every complex manifold admits a K\"{a}hler metric, one can naturally study the Monge-Amp\`{e}re equations \eqref{eq:ma} and \eqref{eq:elliptic} on a general Hermitian manifold. Fu and Yau discussed physical motivation for studying non-K\"{a}hler metrics in a recent paper \cite{FY}.

Cherrier studied \eqref{eq:elliptic} in the general Hermitian setting in the eighties, and showed that in complex dimension $2$ or when $\omega$ is balanced (i.e. $d(\omega^{n-1}) = 0$), there exists a unique normalization of $F$ such that \eqref{eq:elliptic} has a unique solution \cite{Chr}. Precisely, Cherrier proved that under the above conditions, given a smooth function $F$ on $M$, there exists a unique real number $b$ and a unique function $\varphi$ solving the Monge-Amp\`{e}re equation 
\begin{equation}\label{eq:elliptic2}
\log \frac{\det (g_{i\bar{j}} + \partial_i \partial_{\bar{j}} \varphi)}{\det g_{i\bar{j}}} = F + b, \ \ g_{i\bar{j}} + \partial_i \partial_{\bar{j}} \varphi > 0
\end{equation}
such that $\int_M \varphi \ \omega^n = 0$.

Recently, Guan and Li proved that \eqref{eq:elliptic} has a solution on a Hermitian manifold with the added condition $$\partial \bar{\partial} \omega^k = 0$$ for $k=1,2$. They applied this result to finding geodesics in the space of Hermitian metrics. Related work can be found in \cite{BT}, \cite{Chn}, \cite{CLN}, \cite{Do}, \cite{GuB}, \cite{GuP}, \cite{Ma}, \cite{PS}, and \cite{Se}. 

Tosatti and Weinkove gave an alternate proof of Cherrier's result in \cite{TW1}. In a very recent paper \cite{TW2}, they showed that the balanced condition is not necessary and the result holds on a general Hermitian manifold. Dinew and Kolodziej studied \eqref{eq:elliptic} in the Hermitian setting with weaker conditions on the regularity of $F$ \cite{DK}.

In this paper we prove the following theorem.

\begin{theorem}\label{maintheorem}
Let $(M,g)$ be a compact Hermitian manifold of complex dimension $n$ with $\operatorname{Vol}(M) = \int \omega^n = 1$. Let $F$ be a smooth function on $M$. There exists a smooth solution $\varphi$ to the parabolic complex Monge-Amp\`{e}re equation \eqref{eq:ma} for all time. Let 
\begin{equation}\label{eq:tildephi}
\tilde{\varphi} = \varphi - \int_M \varphi \ \omega^n.
\end{equation} Then $\tphi$ converges in $C^\infty$ to a smooth function $\tphi_\infty$. Moreover, there exists a unique real number $b$ such that the pair $(b,\tphi_\infty)$ is the unique solution to \eqref{eq:elliptic2}.
\end{theorem}

We remark that the main theorem gives a parabolic proof of the result due to Tosatti and Weinkove in \cite{TW2}.

The flow \eqref{eq:ma} could be considered as an analogue to K\"{a}hler-Ricci flow for Hermitian manifolds. In the special case that $-\sqrt{-1} \partial \bar{\partial} \log \det g = \sqrt{-1}\partial \bar{\partial} F$ (such an $F$ always exists under the topological condition $c_1^{BC}(M) = 0$, for example) then taking $\sqrt{-1}\partial \bar{\partial}$ of the flow \eqref{eq:ma} yields $$\frac{\partial \omega'}{\partial t} = \sqrt{-1} \partial \bar{\partial} \log \det g'$$ with initial condition $\omega'(0) = \omega$. In general, the right hand side is the first Chern form, but if we assume K\"{a}hler, it becomes $-\operatorname{Ric}(\omega')$.

When $(M,g)$ is K\"{a}hler, Sz\'{e}kelyhidi and Tosatti showed that a weak plurisubharmonic solution to \eqref{eq:elliptic} is smooth using the parabolic flow \eqref{eq:ma} \cite{SzT}. Their result suggests that the flow could be used to prove a similar result in the Hermitian case. In a recent paper \cite{StT}, Streets and Tian consider a different parabolic flow on Hermitian manifolds and suggest geometric applications for the flow. 

We now give an outline of the proof of the main theorem and discuss how it differs from previous results. In sections 2 through 5, we build up theorems that eventually show that $\varphi$ is smooth. Like in Yau's proof, we derive lower order estimates and then apply Schauder estimates to attain higher regularity for the solution.

In Section 2 we use the maximum principle to show that the time derivative of $\varphi$ is uniformly bounded. We define the normalization $$\tphi = \varphi - \int_M \varphi \ \omega^n.$$ We chose to assume that the volume of $M$ is one to simplify the notation of this normalization and the following calculations. Then using the zeroth order estimate from \cite{TW2}, we prove that $\tphi$ is uniformly bounded.

Section 3 contains a proof of the second order estimate. Specifically, we derive that 
\begin{equation}\label{eq:intro1}
\tr_g{g'}  \leq  C_1 e^{C_2 (\sup_{M\times [0,T)}{\tphi} - \inf_{M\times [0,T)} \tphi)} e^{\left( e^{A\left(\sup_{M\times [0,T)}{\tphi}-\inf_{M\times [0,T)}{\tphi}\right)}-e^{A\left(\sup_{M\times [0,T)}{\tphi}-\tphi\right)}\right)}
\end{equation}
where $[0,T)$ is the maximum interval of existence for $\varphi$ and $C_1$, $C_2$, and $A$ are uniform constants. This estimate is not as sharp as the estimate $$\tr_g{g'} \leq C e^{\left(e^{A\left(\sup_{M} \varphi - \inf_{M} \varphi\right)} - e^{A\left(\sup_{M} \varphi - \varphi \right)}\right)}$$ from Guan and Li or the estimate $$\tr_g{g'} \leq C e^{A\left(\varphi - \inf_{M} \varphi\right)}$$ from Tosatti and Weinkove in the case $n = 2$ or $\omega$ balanced. Cherrier also produced a different estimate. These estimates are from the elliptic case, but they suggest that \eqref{eq:intro1} could be improved. The proof of \eqref{eq:intro1} follows along the method of Tosatti and Weinkove in \cite{TW1}, but there are extra terms to control that arrive in the parabolic case. 

In Section 5, we derive a H\"{o}lder estimate for the time dependent metric $g'_{i\bar{j}}$. This estimate provides higher regularity using a method of Evans \cite{Ev} and Krylov \cite{Kr}. To prove the H\"{o}lder estimate, we apply a theorem of Lieberman \cite{L}, a parabolic analogue of an inequality from Trudinger \cite{Tr}. The method follows closely with the proof of the analogous estimate in \cite{TW1}, but differs in controlling the extra terms that arise from the time dependence of $\varphi$. 

We show that $\varphi$ is smooth and also prove the long time existence of the flow \eqref{eq:ma} in Section 5. The proof uses a standard bootstrapping argument.

Section 6 uses analogues of lemmas from Li and Yau \cite{LY} to prove a Harnack inequality for the equation $$\frac{\partial u}{\partial t} = g'^{i\bar{j}} \partial_i \partial_{\bar{j}} u$$ where $g'^{i\bar{j}}\partial_i \partial_{\bar{j}}$ is the complex Laplacian with respect to $g'$. This differs from the equation $$\left(\triangle - q(x,t) - \frac{\partial}{\partial t}\right) u(x,t) = 0$$ considered by Li and Yau, where $\triangle$ is the Laplace-Beltrami operator.

In Section 7, we apply these lemmas to show that time derivative of $\tphi$ decays exponentially. Precisely, we show that $$\left\vert \frac{\partial \tphi}{\partial t} \right\vert \leq C e^{-\eta t}$$ for some $\eta > 0$. From here we show that $\tphi$ converges to a smooth function $\tphi_\infty$ as $t$ tends to infinity. In fact, the convergence occurs in $C^\infty$ and $\tphi_\infty$ is part of the unique pair $(b,\tphi_\infty)$ solving the elliptic Monge-Amp\`{e}re equation $$\log \frac{\det (g_{i\bar{j}} + \partial_i \partial_{\bar{j}} \tphi_\infty)}{\det g_{i\bar{j}}} = F + b$$ where $$b = \int_M \left( \log \frac{\det (g_{i\bar{j}} + \partial_i \partial_{\bar{j}} \tphi_\infty)}{\det g_{i\bar{j}}} - F \right) \ \omega^n.$$ This provides an alternate proof of the main theorem in \cite{TW2}.

\section{Preliminary estimates}

By standard parabolic theory, there exists a unique smooth solution $\varphi$ to \eqref{eq:ma} on a maximal time interval $[0,T),$ where $0 < T \leq \infty$.

We show that the time derivatives of $\varphi$ and its normalization $\tphi$ are bounded. This fact will be used in the second order estimate.

\begin{theorem}\label{theorem:dphidt}
For $\varphi$ a solution of \eqref{eq:ma} and $\tphi$ as in \eqref{eq:tildephi},
\begin{equation}\label{eq:dphidtbound}
\left\vert \frac{\partial \varphi}{\partial t} \right\vert \leq C, \ \ \left\vert \frac{\partial \tphi}{\partial t} \right\vert \leq C
\end{equation}
where $C$ depends only on the initial data.
\end{theorem}
\begin{proof}
Differentiating \eqref{eq:ma} with respect to $t$ gives
\begin{equation}
\frac{\partial \varphi_t}{\partial t} = g'^{i\bar{j}} \partial_i \partial_{\bar{j}} \varphi_t,
\end{equation}
where $\varphi_t = \frac{\partial \varphi}{\partial t}$. So by the maximum principle,
\begin{equation}
\left\vert \frac{\partial \varphi}{\partial t} (x,t) \right\vert \leq C \sup_{x\in M} \left\vert \frac{\partial \varphi}{\partial t} (x,0) \right\vert. 
\end{equation}
From the definition of $\tphi$,
\begin{equation}
\left\vert \frac{\partial \tphi}{\partial t} \right\vert \leq \left\vert \frac{\partial \varphi}{\partial t} \right\vert + \int \left\vert \frac{\partial \varphi}{\partial t} \right\vert \omega^n \leq 2C.
\end{equation}
\end{proof}

We show that $\tphi$ is bounded in $M \times [0,T)$ using the main theorem of \cite{TW2}.

\begin{theorem}
For $\varphi$ a solution to \eqref{eq:ma} and $\tphi$ the normalized solution, there exists a uniform constant $C$ such that $$\sup_{M \times [0,T)}|\tphi| \leq C$$ where $[0,T)$ is the maximum interval of existence for $\varphi$.
\end{theorem}
\begin{proof}
We can rearrange \eqref{eq:ma} to
\begin{equation}
\log{\frac{\det{g'_{i\bar{j}}}}{\det{g_{i\bar{j}}}}} = F - \frac{\partial \varphi}{\partial t}
\end{equation}
Since $\left\vert \frac{\partial \varphi}{\partial t} \right\vert$ is bounded by the previous theorem, this is equivalent to the complex Monge-Amp\`{e}re equation of the main theorem in \cite{TW2}. This implies that $\sup_{M} \varphi(.,t) - \inf_{M} \varphi(.,t) \leq C$ for some $C$ depending only on $(M,g)$ and $F$.

Fix $(x,t)$ in $M\times [0,T)$. Since $\int_M \tilde{\varphi} \ \omega^n = 0$, there exists $(y,t)$ such that $\tilde{\varphi}(y,t) = 0.$ Then
\begin{equation}
\vert \tilde{\varphi}(x,t) \vert = \vert \tilde{\varphi}(x,t) - \tilde{\varphi}(y,t) \vert = \left\vert \varphi(x,t) - \varphi(y,t) \right\vert \leq C.
\end{equation}
Thus $\tilde{\varphi}$ is a bounded function on $M\times [0,T)$.
\end{proof}

\section{The second order estimate}

In this section $\triangle = g^{i\bar{j}} \partial_i \partial_{\bar{j}}$ will denote the complex Laplacian corresponding to $g$. Similarly, write $\triangle' = g'^{i\bar{j}} \partial_i \partial_{\bar{j}}$ for the complex Laplacian for the time dependent metric $g'$. We prove an estimate on $\tr_g{g'} = g^{i\bar{j}} g'_{i\bar{j}} = n + \triangle \tphi$.

\begin{theorem}
For $\varphi$ a solution to \eqref{eq:ma} and $\tilde{\varphi}$ the normalized solution, we have the following estimate $$ \tr_g{g'}  \leq  C_1 e^{C_2 (\sup_{M\times [0,T)}{\tphi} - \inf_{M\times [0,T)} \tphi)} e^{\left( e^{A\left(\sup_{M\times [0,T)}{\tphi}-\inf_{M\times [0,T)}{\tphi}\right)}-e^{A\left(\sup_{M\times [0,T)}{\tphi}-\tphi\right)}\right)}$$ where $[0,T)$ is the maximum interval of existence for $\varphi$ and $C_1$, $C_2$, and $A$ are uniform constants. Hence there exists a uniform constant $C$ such that $\tr_g{g'} \leq C$ and also $$\frac{1}{C} g \leq g' \leq C g.$$
\end{theorem}

\begin{proof}
This proof follows along with the notation and method featured in \cite{TW1}. For brevity we omit some of the calculations and refer the reader to \cite{TW1} and \cite{GL}. Let $E_1$ and $E_2$ denote error terms of the form $$\vert E_1 \vert \leq C_1 \tr_{g'}{g}$$ $$\vert E_2 \vert \leq C_2 (\tr_{g'}{g}) (\tr_{g}{g'})$$ where $C_1$ and $C_2$ are constants depending only on the initial data. We call such a constant depending only on $(M,g)$ and $\sup_M F$ a uniform constant. We remark that by the flow equation \eqref{eq:ma} and estimate \eqref{eq:dphidtbound}, an error term of type $E_1$ is also of type $E_2$ and a uniform constant is of type $E_1$. In general, $C$ will denote a uniform constant whose definition may change from line to line. For a function $\varphi$ on $M$, we write $\varphi_i$ for the ordinary derivative $$\varphi_i = \partial_i \varphi.$$ Similarly, $\varphi_t$ will denote the time derivative of $\varphi$. If $f$ is a function on $M$, we write $\partial f$ for the vector of ordinary derivatives of $f$.

We define the quantity
\begin{equation}
Q = \log{\tr_g{g'} + e^{A\left(\sup_{M\times [0,T)}{\tphi} - \tphi\right)}}
\end{equation}
We note that the form of $Q$ differs here than in \cite{TW1} and Yau's estimate \cite{Yau} and Aubin's estimate \cite{Au}. They consider a quantity of the form $\log{\tr_g{g'}} - A\varphi$. The exponential in the definition of $Q$ helps to control a difficult term in the analysis.

Fix $t'\in [0,T)$. Then let $(x_0,t_0)$ be the point in $M\times [0,t']$ where $Q$ attains its maximum. Notice that if $t_0 = 0$ the result is immediate, so we assume $t_0 > 0$. To start the proof, we need to perform a change of coordinates made possible by the following lemma from \cite{GL}.
\begin{lemma}\label{lem:coord}
There exists a holomorphic coordinate system centered at $x_0$ such that for all $i$ and $j$,
\begin{equation}\label{eq:coord}
g_{i\bar{j}}(x_0) = \delta_{ij}, \ \ \partial_j g_{i\bar{i}}(x_0) = 0,
\end{equation} 
and also such that the matrix $\varphi_{i\bar{j}}(x_0,t_0)$ is diagonal.
\end{lemma} 
Applying $\triangle' - \frac{\partial}{\partial t}$ to $Q$,
\begin{eqnarray}\label{eq:boxq1}
\left(\triangle' - \frac{\partial}{\partial t}\right) Q & = & \frac{\triangle' \tr_g{g'}}{\tr_g{g'}} - \frac{\vert \partial \tr_g{g'} \vert^2_{g'}}{(\tr_g{g'})^2} - \frac{\triangle \frac{\partial \varphi}{\partial t}}{\tr_g{g'}} + A \frac{\partial \tphi}{\partial t} e^{A(\sup_{M\times [0,T)}{\tphi} - \tphi)} \\
& & \ \ + \triangle' e^{A(\sup_{M\times [0,T)}{\tphi} - \tphi)}.
\end{eqnarray}
First we will control the first and third terms in \eqref{eq:boxq1} simultaneously. We apply the complex Laplacian $\triangle$ to the complex Monge-Amp\`{e}re equation:
\begin{eqnarray}
\triangle \frac{\partial \varphi}{\partial t} & = & -g^{k\bar{l}}g'^{p\bar{j}}g'^{i\bar{q}}\partial_k g'_{p\bar{q}} \partial_{\bar{l}} g'_{i\bar{j}} + g^{k\bar{l}}g'^{i\bar{j}} \partial_k \partial_{\bar{l}} g'_{i\bar{j}} + g^{k\bar{l}}g^{p\bar{j}}g^{i\bar{q}}\partial_k g_{p\bar{q}} \partial_{\bar{l}} g_{i\bar{j}} \nonumber \\ 
& & \ \ \ - g^{k\bar{l}}g^{i\bar{j}} \partial_k \partial_{\bar{l}} g_{i\bar{j}} - \triangle F \nonumber \\ 
\label{eq:eqn3} & = & \sum_{i,k} g'^{i\bar{i}} \varphi_{i\bar{i}k\bar{k}} - \sum_{i,j,k} g'^{i\bar{i}}g'^{j\bar{j}}\partial_k g'_{i\bar{j}} \partial_{\bar{k}} g'_{j\bar{i}} + E_1.
\end{eqnarray}
For the first term in \eqref{eq:boxq1}, following a calculation in \cite{TW1} (see equation (2.6) in \cite{TW1}) gives
\begin{equation}\label{eq:eqn4}
\triangle' tr_g{g'}  =  \sum_{i,k} g'^{i\bar{i}} \varphi_{i\bar{i}k\bar{k}} - 2\operatorname{Re}{\left(\sum_{i,j,k} g'^{i\bar{i}}\partial_{\bar{i}}g_{j\bar{k}}\varphi_{k\bar{j}i} \right)} + E_2.
\end{equation}
We will now handle the $2\operatorname{Re}{\left(\sum_{i,j,k} g'^{i\bar{i}}\partial_{\bar{i}}g_{j\bar{k}}\varphi_{k\bar{j}i} \right)}$ term in \eqref{eq:eqn4} using a trick from \cite{GL}. Using Lemma 3.2, at the point $(x_0,t_0)$,
\begin{equation}
\sum_{i,j,k} g'^{i\bar{i}}\partial_{\bar{i}}g_{j\bar{k}}\varphi_{k\bar{j}i} =  \sum_i \sum_{j \neq k} g'^{i\bar{i}}\partial_{\bar{i}}g_{j\bar{k}}\partial_k g'_{i\bar{j}} + E_1.
\end{equation}
Hence,
\begin{eqnarray}
\left\vert 2\operatorname{Re}{\left(\sum_{i,j,k} g'^{i\bar{i}}\partial_{\bar{i}}g_{j\bar{k}}\varphi_{k\bar{j}i} \right)} \right\vert & \leq & \sum_i \sum_{j \neq k} g'^{i\bar{i}} g'^{j\bar{j}} \partial_k g'_{i\bar{j}} \partial_{\bar{k}} g'_{j\bar{i}} + \sum_i \sum_{j \neq k} g'^{i\bar{i}}g'_{j\bar{j}} \partial_{\bar{i}}g_{j\bar{k}} \partial_i g_{k\bar{j}} + E_1 \nonumber \\
\label{eq:eqn5}& \leq & \sum_i \sum_{j \neq k} g'^{i\bar{i}} g'^{j\bar{j}} \partial_k g'_{i\bar{j}} \partial_{\bar{k}} g'_{j\bar{i}} + E_2.
\end{eqnarray}
Putting together \eqref{eq:eqn3}, \eqref{eq:eqn4}, and \eqref{eq:eqn5} gives
\begin{equation}\label{eq:eqn6}
\triangle' \tr_g{g'} - \triangle \frac{d\varphi}{dt} \geq \sum_{i,j} g'^{i\bar{i}} g'^{j\bar{j}} \partial_j g'_{i\bar{j}} \partial_{\bar{j}} g'_{j\bar{i}} + E_2.
\end{equation}

Now we will control the $\frac{\vert \partial \tr_g{g'} \vert^2_{g'}}{(\tr_g{g'})^2}$ term in \eqref{eq:boxq1}. By Lemma 3.2 we have at $(x_0,t_0)$,
\begin{equation}\label{eq:eqn2121}
\partial_i \tr_g{g'} = \partial_i \triangle \varphi = \partial_i \sum_j \varphi_{j\bar{j}} = \sum_j \partial_j \varphi_{i\bar{j}} = \sum_j \partial_j g'_{i\bar{j}} - \sum_j \partial_j g_{i\bar{j}}.
\end{equation}
So
\begin{equation}\label{eq:eqn7}
\frac{\vert \partial \tr_g{g'} \vert^2_{g'}}{\tr_g{g'}} = \frac{1}{tr_g{g'}} \sum_{i,j,k} g'^{i\bar{i}} \partial_j g'_{i\bar{j}} \partial_{\bar{k}} g'_{k\bar{i}} - \frac{2}{tr_g{g'}} \operatorname{Re}{\left(\sum_{i,j,k}g'^{i\bar{i}}\partial_j g_{i\bar{j}} \partial_{\bar{k}} g'_{k\bar{i}} \right)} + E_1.
\end{equation}
As in Yau's second order estimate, we use Cauchy-Schwarz on the first term in \eqref{eq:eqn7} (see \cite{TW1} equation (2.15) for the exact calculation).
\begin{equation}\label{eq:eqn8}
\frac{1}{tr_g{g'}} \sum_{i,j,k} g'^{i\bar{i}} \partial_j g'_{i\bar{j}} \partial_{\bar{k}} g'_{k\bar{i}} \leq \sum_{i,j} g'^{i\bar{i}} g'^{j\bar{j}} \partial_j g'_{i\bar{j}} \partial_{\bar{j}} g'_{j\bar{i}}.
\end{equation}
To deal with the second term in \eqref{eq:eqn7}, since $(x_0,t_0)$ is the maximum point of $Q$, $\partial_{\bar{i}} Q = 0$ implies
\begin{equation}\label{eq:eqn2424}
\frac{1}{\tr_g{g'}} \sum_k \partial_{\bar{i}} g'_{k\bar{k}} = A \partial_{\bar{i}} \varphi e^{A(\sup_{M\times [0,T)}{\tphi} - \tphi)}.
\end{equation}
Using equations \eqref{eq:eqn2424} and \eqref{eq:eqn2121}  we can bound the difficult term:
\begin{align}
&\left\vert \frac{2}{\tr_g{g'}} \operatorname{Re}{\left(\sum_{i,j,k}g'^{i\bar{i}}\partial_j g_{i\bar{j}} \partial_{\bar{k}} g'_{k\bar{i}} \right)} \right\vert \nonumber \\
&\ \ \ =\left\vert \frac{A}{\tr_g{g'}} e^{A(\sup_{M\times [0,T)}{\tphi}-\tphi)} 2\operatorname{Re}{\left(\sum_{i,j,k} g'^{i\bar{i}}\partial_j g_{i\bar{j}} \partial_{\bar{i}}\varphi\right)}\right\vert + E_1 \nonumber \\
&\ \ \ \leq A^2 \vert \partial \varphi \vert^2_{g'} e^{A(\sup_{M\times [0,T)}{\tphi}-\tphi)}+\frac{C(\tr_{g'}{g})}{(\tr_g{g'})^2} e^{A(\sup_{M\times [0,T)}{\tphi}-\tphi)} + E_1 \nonumber \\
\label{eq:eqn9}&\ \ \ \leq A^2 \vert \partial \varphi \vert^2_{g'} e^{A(\sup_{M\times [0,T)}{\tphi}-\tphi)}+ C(\tr_{g'}{g})e^{A(\sup_{M\times [0,T)}{\tphi}-\tphi)} + E_1,
\end{align}
where for the last inequality we used the fact that $\tr_g{g'}$ is bounded from below away from zero by the flow equation \eqref{eq:ma} and estimate \eqref{eq:dphidtbound}.

Plugging \eqref{eq:eqn8} and \eqref{eq:eqn9} into \eqref{eq:eqn7} gives 
\begin{eqnarray}\label{eq:eqn10}
\frac{\vert \partial \tr_g{g'} \vert^2_{g'}}{(\tr_g{g'})^2} & \leq & \frac{1}{(tr_g{g'})^2} \sum_{i,j} g'^{i\bar{i}}g'^{j\bar{j}}\partial_jg'_{i\bar{j}}\partial_{\bar{j}}g'_{j\bar{i}} + A^2\vert \partial \varphi \vert^2_{g'} e^{A(\sup_{M\times [0,T)}{\tphi}-\tphi)} \nonumber \\
& & \ \ \ +C(\tr_{g'}g)e^{A(\sup_{M\times [0,T)}{\tphi}-\tphi)} + E_1.
\end{eqnarray}
By combining \eqref{eq:eqn6} and \eqref{eq:eqn10} with \eqref{eq:boxq1} at the point $(x_0,t_0)$, we get the inequality
\begin{eqnarray}
0 & \geq & \frac{1}{tr_g{g'}} \left(\sum_{i,j} g'^{i\bar{i}} g'^{j\bar{j}} \partial_j g'_{i\bar{j}} \partial_{\bar{j}} g'_{j\bar{i}} + E_2 \right) - \frac{1}{\tr_g{g'}} \sum_{i,j} g'^{i\bar{i}} g'^{j\bar{j}} \partial_j g'_{i\bar{j}} \partial_{\bar{j}} g'_{j\bar{i}} \nonumber \\
& & \ \ \ - A^2\vert \partial \varphi \vert^2_{g'} e^{A(\sup_{M\times [0,T)} \tphi - \tphi)} - \tr_{g'}g e^{A(\sup_{M\times [0,T)} \tphi - \tphi)} + E_1 \nonumber \\
& & \ \ \ + A\frac{\partial \tphi}{\partial t}e^{A(\sup_{M\times [0,T)} \tphi - \tphi)} + \left(-An + A\tr_{g'}g + A^2 \left\vert\partial \varphi \right\vert^2_{g'} \right)e^{A(\sup_{M\times [0,T)}{\tphi}-\tphi)} \nonumber \\
& \geq & -A(C + n)e^{A(\sup_{M\times [0,T)}{\tphi} - \tphi)} + (A-1) \tr_{g'}{g}e^{A(\sup_{M\times [0,T)}{\tphi} - \tphi)}  - C_1 \tr_{g'}{g} \nonumber \\
& \geq & -A(C+n)e^{A(\sup_{M\times [0,T)}{\tphi} - \tphi)} + \left(A-1-C_1\right)\tr_{g'}g.
\end{eqnarray}
Taking $A$ large enough so that $$\left(A-1-C_1\right) > 0$$ implies that at $(x_0, t_0)$,
\begin{equation}\label{eq:eqn12}
\tr_{g'}{g}(x_0,t_0) \leq Ce^{A(\sup_{M\times [0,T)}{\tphi} - \inf_{M\times [0,T)} \tphi)}.
\end{equation}
Then 
\begin{eqnarray}\label{eq:eqn13}
\tr_g{g'}(x_0,t_0) & \leq & \frac{1}{(n-1)!}\left(\tr_{g'}{g}\right)^{n-1}\frac{\det{g'}}{\det{g}} \nonumber \\
& = & \frac{1}{(n-1)!} \left(\tr_{g'}{g}\right)^{n-1} e^{F-\frac{\partial \varphi}{\partial t}} \nonumber \\
& \leq & Ce^{A(n-1)(\sup_{M\times [0,T)}{\tphi} - \inf_{M\times[0,T)} \tphi)}
\end{eqnarray}
For all $(x,t)$ in $M\times [0,t']$ 
\begin{eqnarray}
\log{\tr_{g}{g'}(x,t)} & + & e^{A(\sup_{M\times [0,T)}{\tphi} - \tphi(x,t))} \\ 
& \leq & \log \left(Ce^{A(n-1)(\sup_{M\times [0,T)}{\tphi} - \inf_{M\times [0,T)}\tphi)}\right) + e^{A(\sup_{M\times [0,T)}{\tphi}-\inf_{M\times [0,T)}{\tphi})} \nonumber
\end{eqnarray}
\begin{equation}
\tr_g{g'} \leq C_1 e^{C_2(\sup_{M\times [0,T)}{\tphi} - \inf_{M\times [0,T)} \tphi)} e^{\left( e^{A\left(\sup_{M\times [0,T)}{\tphi}-\inf_{M\times [0,T)}{\tphi}\right)}-e^{A\left(\sup_{M\times [0,T)}{\tphi}-\tphi\right)}\right)}.
\end{equation}
\end{proof}

\section{The H\"{o}lder estimate for the metric}

The estimate in this section is local, so it suffices to work in a domain in $\mathbb{C}^n$. To fix some notation, define the parabolic distance function between two points $(x,t_1)$ and $(y,t_2)$ in $\mathbb{C}^n \times [0,T)$ to be $\left\vert(x,t_1) - (y,t_2)\right\vert = \max(|x - y| , |t_1 - t_2|^{1/2})$.

For a domain $\Omega \in \mathbb{C}^n \times [0,T)$ and a real number $\alpha \in (0,1)$, define for a function $\varphi$ on $\mathbb{C}^n \times [0,T)$,
\begin{equation} \nonumber
[\varphi]_{\alpha,(x_0,t_0)} = \sup_{(x,t) \in \Omega \setminus \{(x_0,t_0)\}} \frac{|\varphi(x,t) - \varphi(x_0,t_0)|}{\left\vert(x,t)-(x_0,t_0)\right\vert^{\alpha}}
\end{equation}
and
\begin{equation}
[\varphi]_{\alpha,\Omega} = \sup_{(x,t)\in \Omega} [\varphi]_{\alpha,(x,t)}.
\end{equation}

We will show that $$[g'_{i\bar{j}}]_{\alpha,\Omega} \leq C$$ for an appropriate choice of $\Omega$. The smoothness of $\varphi$ and $\tphi$ will follow. Given the H\"{o}lder bound for the metric and the second order estimate for $\tphi$, we can differentiate the flow and apply Schauder estimates to achieve higher regularity.

\begin{theorem}
Let $\varphi$ be a solution to the flow \eqref{eq:ma} and $g'_{i\bar{j}} = g_{i\bar{j}} + \varphi_{i\bar{j}}$. Fix $\varepsilon > 0$. Then there exists $\alpha \in (0,1)$ and a constant $C$ depending only on the initial data and $\varepsilon$ such that
\begin{equation}\label{eq:holderThm1}
[g'_{i\bar{j}}]_{\alpha,\Omega} \leq C
\end{equation}
where $\Omega = M \times [\varepsilon, T)$. 
\end{theorem}

We apply a method due to Evans \cite{Ev} and Krylov \cite{Kr}. The proof itself is essentially contained in \cite{L} and \cite{GT}, but only in the case where the manifold is $\mathbb{R}^n$. Hence we produce a self-contained proof in the notation of this problem. The method of this proof follows closely with the analogous estimate in \cite{TW1} and \cite{Si}. The main issue is applying the correct Harnack inequality to get the estimate.

\begin{proof}

Let $B \in \mathbb{C}^n$ be an open ball about the origin. Fix a point $t_0 \in [\varepsilon, T)$. To prove \eqref{eq:holderThm1} it suffices to show that for sufficiently small $R > 0$ there exists a uniform $C$ and $\delta > 0$ such that 
$$\sum_{i=1}^n \operatorname{osc}_{Q(R)}(\varphi_{\gamma_i\bar{\gamma_i}}) + \operatorname{osc}_{Q(R)}(\varphi_t) \leq C R^\delta$$ 
where $\{\gamma_i\}$ is a basis for $\mathbb{C}^n$ and $Q(R)$ is the parabolic cylinder $$Q(R) = \{ (x,t) \in B \times [0,T) | |x| \leq R, t_0 - R^2 \leq t \leq t_0 \}.$$
We rewrite the flow as
\begin{equation}\label{eq:h1}
\frac{\partial \varphi}{\partial t} = \log \det g'_{i\bar{j}} + H
\end{equation}
where $H = -\log \det g_{i\bar{j}} - F$. We define the operator $\Phi$ on a matrix $A$ by $$\Phi(A) = \log \det A,$$
then \eqref{eq:h1} becomes
\begin{equation}\label{eq:h2}
\frac{\partial \varphi}{\partial t} = \Phi (g') + H.
\end{equation}
By the concavity of $\Phi$, for all $(x,t_1)$ and $(y,t_2)$ in $B \times [0,T)$,
\begin{equation} \nonumber
\sum \frac{\partial \Phi}{\partial a_{i\bar{j}}}(g'(y,t_2)) \left(g'_{i\bar{j}}(x,t_1) - g'_{i\bar{j}}(y,t_2)\right) \geq  \frac{\partial \varphi}{\partial t}(x,t_1) - \frac{\partial \varphi}{\partial t}(y,t_2) - H(x) + H(y). \end{equation}
The Mean Value Theorem applied to $H$ shows that
\begin{equation}\label{eq:h5}
\frac{\partial \varphi}{\partial t}(x,t_1) - \frac{\partial \varphi}{\partial t}(y,t_2) + \sum \frac{\partial \Phi}{\partial a_{i\bar{j}}}(g'(y,t_2))\left(g'_{i\bar{j}}(y,t_2) - g'_{i\bar{j}}(x,t_1)\right) \leq C|x-y|.
\end{equation}

Now we must recall a lemma from linear algebra.

\begin{lemma}\label{lem:LA}
There exists a finite number $N$ of unit vectors $\gamma_{\nu} = (\gamma_{\nu 1}, \ldots, \gamma_{\nu n}) \in \mathbb{C}^n$ and real-valued functions $\beta_{\nu}$ on $B \times [0,T)$, for $\nu = 1,2,\ldots,N$ with
\begin{eqnarray}
& (i) & 0 < C_1 \leq \beta_{\nu} \leq C_2 \\ \nonumber
& (ii) & \gamma_1, \ldots, \gamma_N \ containing \ an \ orthonormal \ basis \ of \ \mathbb{C}^n \nonumber
\end{eqnarray}
such that
$$\frac{\partial \Phi}{\partial a_{i\bar{j}}}\left(g'(y,t_2)\right) = \sum_{\nu = 1}^N \beta_{\nu}(y,t_2) \gamma_{\nu i} \overline{\gamma_{\nu j}}.$$
\end{lemma}
We define for $\nu = 1,\ldots,N$, $$w_{\nu} = \partial_{\gamma_{\nu}} \partial_{\bar{\gamma_{\nu}}} \varphi = \sum_{i,j=1}^n \gamma_{\nu i} \overline{\gamma_{\nu j}} \varphi_{i\bar{j}}.$$
We write $w_0 = -\frac{\partial \varphi}{\partial t}$ and $\beta_0 = 1$. Then using the linear algebra lemma, \eqref{eq:h5} can be rewritten as
\begin{equation}\label{eq:h6}
\sum_{\nu = 0}^N \beta_{\nu}(y,t_2) \left(w_{\nu}(y,t_2) - w_{\nu}(x,t_1)\right) \leq C|x-y|.
\end{equation}
Letting $\gamma$ be an arbitrary unit vector in $\mathbb{C}^n$, we differentiate the flow \eqref{eq:ma} along $\gamma$ and $\bar{\gamma}$:
\begin{eqnarray} \nonumber
\frac{\partial \varphi_{\gamma \bar{\gamma}}}{\partial t} & = & \frac{\partial^2 \Phi}{\partial a_{i\bar{j}} \partial a_{k\bar{l}}} (g') g'_{i\bar{j}\gamma} g'_{k\bar{l}\bar{\gamma}} + \frac{\partial \Phi}{\partial a_{i\bar{j}}} (g') g'_{i\bar{j} \gamma \bar{\gamma}} + H_{\gamma \bar{\gamma}} \\ \label{eq:h3}
& \leq & g'^{i\bar{j}} g'_{i\bar{j} \gamma \bar{\gamma}} + H_{\gamma \bar{\gamma}}
\end{eqnarray}
where on the last line we used the concavity of $\Phi$ and the fact that $\frac{\partial \Phi}{\partial a_{i\bar{j}}}(g') = g'^{i\bar{j}}$. Applying $\frac{\partial}{\partial t}$ to \eqref{eq:h2} gives
\begin{eqnarray}\label{eq:h4}
\frac{\partial \varphi_t}{\partial t}  = g'^{i\bar{j}} \varphi_{i\bar{j} t}.
\end{eqnarray}
From \eqref{eq:h3} and \eqref{eq:h4} we have a bounded function $h$ (depending on $g'^{i\bar{j}}$ which is bounded by Theorem 3.1) such that
\begin{equation}\label{eq:h7}
-\frac{\partial w_{\nu}}{\partial t} + g'^{i\bar{j}} \partial_i \partial_{\bar{j}} w_{\nu} \geq h.
\end{equation}

Recall that $t_0$ is a fixed point in $[\varepsilon, T)$. Pick $R > 0$ small enough such that $t_0 - 5R^2 > t_0 / 2$. We define another parabolic cylinder $$\Theta(R) = \{(x,t) \in B \times [0,T) | |x|<R, t_0 - 5R^2 \leq t \leq t_0 - 4R^2 \}.$$
For $s = 1, 2$ and $\nu = 0, 1, \ldots, N$, let $$M_{s\nu} = \sup_{Q(sR)} w_{\nu}, \ m_{s\nu} = \inf_{Q(sR)} w_{\nu},$$ and $$\psi(sR) = \sum_{\nu = 0}^N \left(M_{s\nu} - m_{s\nu}\right).$$

We let $l$ be an integer such that $0 \leq l \leq N$ and $v = M_{2l} - w_l$. To continue we need Theorem 7.37 from \cite{L}. We say that $v \in W^{2,1}_{2n+1}$ if $v_x, v_{ij}, v_{i\bar{j}}, v_{\bar{i} \bar{j}},$ and $v_t$ are in $L^{2n+1}$. We restate the theorem as follows.

\begin{lemma}
Suppose that $v(x,t) \in W^{2,1}_{2n+1}$ satisfies $$-\frac{\partial v}{\partial t} + g'^{i\bar{j}}\partial_i \partial_{\bar{j}} v \leq f$$ and $v \geq 0$ on $Q(4R)$. Then there exists a constant $C$ and a $p > 0$ depending only on the bounds of $g'^{i\bar{j}}$ and the eigenvalues of $g'^{i\bar{j}}$ such that
$$\frac{1}{R^{2n+2}} \left( \int_{\Theta(R)} v^p \right)^{1/p} \leq C \left( \inf_{Q(R)} v + R^{\frac{2n}{2n+1}} ||f||_{n+1} \right).$$
\end{lemma}

Since $v$ satisfies $-\frac{\partial v}{\partial t} + g'^{i\bar{j}}\partial_i \partial_{\bar{j}} v \leq -h,$ we can apply the Harnack inequality to get
\begin{equation}\label{eq:h8}
\frac{1}{R^{2n+2}} \left( \int_{\Theta(R)} \left( M_{2l} - w_{l} \right)^p \right)^{1/p} \leq C \left( M_{2l} - M_l + R^{\frac{2n}{2n+1}}\right).
\end{equation}
For every $(x,t_1)$ and $(y,t_2)$ in $Q(2R)$, \eqref{eq:h6} gives
$$\beta_l(y,t_2)\left( w_l(y,t_2) - w_l(x,t_1) \right) \leq CR + \sum_{\nu \neq l} \beta_{\nu} \left(w_{\nu}(x,t_1) - w_{\nu}(y,t_2) \right).$$
The definition of $m_{2l}$ allows us to choose $(x,t_1)$ in $Q(2R)$ such that $w_l(x,t_1) \leq m_{2l} + \varepsilon$. Since $\varepsilon$ is arbitrary,
$$w_l(y,t_2) - m_{2l} \leq CR + C_2 \sum_{\nu \neq l} \left( M_{2\nu} - w_{\nu}(y,t_2) \right).$$
After integrating over $\Theta(R)$ and applying \eqref{eq:h8}, we have
\begin{eqnarray}\nonumber
\frac{1}{R^{2n+2}} \left( \int_{\Theta(R)} \left( w_l - m_{2l} \right)^p \right)^{1/p} & \leq & \frac{1}{R^{2n+2}} \left( \int_{\Theta(R)} \left( CR + C_2 \sum_{\nu \neq l} \left( M_{2\nu} - w_{\nu} \right) \right)^p \right)^{1/p} \\ \nonumber
& \leq & C_3 R + C_4 \sum_{\nu \neq l} \frac{1}{R^{2n+2}} \left( \int_{\Theta(R)} \left(M_{2\nu} - w_{\nu} \right)^p \right)^{1/p} \\ \label{eq:h9}
& \leq & C_5 \sum_{\nu \neq l} \left(M_{2\nu} - M_{\nu} \right) + C_6 R^{\frac{2n}{2n+1}}
\end{eqnarray}
where on the last line we used the fact that $R < 1$ is small. Adding \eqref{eq:h8} and \eqref{eq:h9} yields
\begin{eqnarray}
M_{2l} - m_{2l} & \leq & C_7 \sum_{\nu = 0}^N \left( M_{2\nu} - M_{\nu} \right) + C_8 R^{\frac{2n}{2n+1}} \\ \nonumber
& \leq & C_7 \sum_{\nu = 0}^N \left( M_{2\nu} - M_{\nu} + m_{\nu} - m_{2\nu} \right) + C_8 R^{\frac{2n}{2n+1}} \\ \nonumber
& = & C_7 \left( \psi(2R) - \psi(R) \right) + C_8 R^{\frac{2n}{2n+1}}. \nonumber
\end{eqnarray}
Summing over $l$ shows that
$$\psi(2R) \leq C_9 \left( \psi(2R) - \psi(R) \right) + C_{10} R^{\frac{2n}{2n+1}}$$
and thus for some $0 < \lambda < 1$,
$$\psi(R) \leq \lambda \psi(2R) + C_{11} R^{\frac{2n}{2n+1}}.$$
Applying a standard iteration argument (see Chapter 8 in \cite{GT}) shows that $$\psi(R) \leq C R^\delta$$ for some $\delta > 0$, completing the proof.
\end{proof}

\section{Long time existence and smoothness of the normalized solution}

In this section we show that the solution $\varphi$ and its normalization $\tphi$ are smooth and exist for all time, hence proving part of the main theorem. The proof uses a standard bootstrapping argument.

\begin{theorem}
Let $(M,g)$ be a Hermitian manifold and $F$ a smooth function on $M$. Let $\varphi$ be a solution to the flow $$\frac{\partial \varphi}{\partial t} = \log \frac{\det (g_{i\bar{j}} + \varphi_{i\bar{j}})}{\det (g_{i\bar{j}})} - F$$ and let $\tilde{\varphi} = \varphi - \int_M \varphi \ \omega^n.$ Then there are uniform $C^\infty$ estimates for $\tilde{\varphi}$ on $[0,T)$. Moreover, $T = \infty$. 
\end{theorem}

\begin{proof}
Differentiating the flow with respect to $z^k$ gives
\begin{equation}
\frac{\partial \varphi_k}{\partial t} = g'^{i\bar{j}} \partial_i \partial_{\bar{j}} \varphi_k - F_k - \frac{\partial}{\partial z^k} \log \det g_{i\bar{j}}.  
\end{equation}

The second order estimate estimate implies that the above equation is uniformly parabolic. Theorem 4.1 shows that the coefficients in the above equation are H\"older continuous with exponent $\alpha$. Using the Schauder estimate (see Theorem 4.9 in \cite{L}, for example) gives a uniform parabolic $C^{2+\alpha}$ bound on $\varphi_k$. Similarly, one obtains a uniform parabolic $C^{2+\alpha}$ estimate for $\varphi_{\bar{k}}$. 

But the better differentiability of $\varphi$ allows us to differentiate the flow again and obtain a uniformly parabolic equation with H\"older continuous coefficients. The Schauder estimate will give a uniform parabolic $C^{2+\alpha}$ estimate on $\varphi_{kl}, \varphi_{k\bar{l}}$, and $\varphi_{\bar{k}\bar{l}}$. Repeated application shows that $\tilde{\varphi}$ is uniformly bounded in $C^\infty$. Hence $\tphi$ and thus $\varphi$ are smooth. We note that $\varphi$ is not necessarily bounded in $C^0$. The above iterations only provide regularity for the derivatives of $\varphi$.

To see that $T = \infty$, suppose that for $T < \infty$, $[0,T)$ is the maximal interval for the existence of the solution. Since $\tilde{\varphi}$ is smooth, we can apply short time existence to extend the flow for $\tilde{\varphi}$ to $[0,T+\varepsilon)$, a contradiction.

\end{proof}

\section{The Harnack inequality}

We begin this section by proving lemmas analogous to those of Li and Yau \cite{LY} for the equation $\frac{\partial u}{\partial t} = g'^{i\bar{j}}\partial_i \partial_{\bar{j}} u$ for a positive function $u$ on a Hermitian manifold (see \cite{W} for the proof of these lemmas in the K\"{a}hler case). The lemmas lead to a Harnack inequality, which in turn shows that the time derivative of $\tphi$ decays exponentially. This allows us to prove the convergence of $\tphi$ as $t$ tends to infinity.

In this section, we again use the notation $u_t = \frac{\partial u}{\partial t}$ and $u_i = \partial_i u$ for the ordinary derivatives of a function $u$ on $M$.

Let $u$ be a positive function on $M$. Consider the heat type equation $$u_t = g'^{i\bar{j}} u_{i\bar{j}}$$ where $g'_{i\bar{j}}$ denotes the time dependent metric $g_{i\bar{j}} + \varphi_{i\bar{j}}$. Define $\tilde{\varphi} = \varphi - \int_M \varphi \omega^n$.

Define $f = \log{u}$ and $F = t(|\partial f|^2 - \alpha f_t)$ where $1 < \alpha < 2$. We remark that this $F$ is different from the one in equation \eqref{eq:ma}. Then $$g'^{i\bar{j}}f_{i\bar{j}} - f_t = -|\partial f|^2$$ where $\partial f$ is the vector containing the ordinary derivatives of $f$ and $$|\partial f|^2 = g'^{i\bar{j}} \partial_i f \partial_{\bar{j}} f.$$ Also write $$\left< X,Y \right> = g'^{i\bar{j}} X_i Y_{\bar{j}}$$ for the inner product of two vectors $X$ and $Y$ with respect to $g'_{i\bar{j}}$.

We now prove an estimate that will be useful in applying the maximum principle to $F$.

\begin{lemma} There exist constants $C_1$ and $C_2$ depending only on the bounds of the metric $g'$ such that for $t > 0$, 
$$g'^{k\bar{l}} F_{k\bar{l}} - F_t \geq \frac{t}{2n}\left( |\partial f|^2 - f_t \right)^2 - 2\operatorname{Re} \left< \partial f, \partial F \right> - \left(|\partial f|^2 - \alpha f_t\right) - C_1 t|\partial f|^2 - C_2t.$$
\end{lemma}

\begin{proof}
First calculate $F = -t g'^{i\bar{j}}f_{i\bar{j}} - t(\alpha - 1) f_t$. Then 
\begin{equation}\label{eq:timederiv}
(g'^{i\bar{j}} f_{i\bar{j}})_t = \frac{1}{t^2} F - \frac{1}{t} F_t - (\alpha - 1) f_{tt}
\end{equation}
and 
\begin{eqnarray} \nonumber
F_t & = & |\partial f|^2 - \alpha f_t + t \left(g'^{i\bar{j}} f_{ti} f_{\bar{j}} + g'^{i\bar{j}} f_i f_{t\bar{j}} + \left(\frac{\partial}{\partial t} g'^{i\bar{j}}\right) f_i f_{\bar{j}} - \alpha f_{tt} \right) \\ \label{eq:timederiv2}
& = & |\partial f|^2 - \alpha f_t + 2t \operatorname{Re} \left< \partial f, \partial (f_t) \right> + t \left(\frac{\partial}{\partial t} g'^{i\bar{j}}\right) f_i f_{\bar{j}} - \alpha t f_{tt}.
\end{eqnarray}
We calculate $g'^{k\bar{l}}F_{k\bar{l}}$ to get the desired estimate.
\begin{eqnarray}\nonumber
g'^{k\bar{l}} F_{k\bar{l}} & = & t g'^{k\bar{l}} \Big[ \left(g'^{i\bar{j}}\right)_{k\bar{l}} f_i f_{\bar{j}} + \left(g'^{i\bar{j}}\right)_k f_{i\bar{l}} f_{\bar{j}} + \left( g'^{i\bar{j}} \right)_k f_i f_{\bar{j}\bar{l}} + \left( g'^{i\bar{j}} \right)_{\bar{l}} f_{ik} f_{\bar{j}} + g'^{i\bar{j}} f_{ik\bar{l}} f_{\bar{j}} \\ \label{eq:lapF}
& & \ \ + g'^{i\bar{j}} f_{ik} f_{\bar{j}\bar{l}} + \left(g'^{i\bar{j}}\right)_{\bar{l}} f_i f_{\bar{j}k} + g'^{i\bar{j}} f_{i\bar{l}} f_{\bar{j}k} + g'^{i\bar{j}} f_i f_{\bar{j}k\bar{l}} - \alpha f_{tk\bar{l}} \Big].
\end{eqnarray}
Now we control all of the above terms using the bounds on the metric obtained in Theorem 3.1 and the higher order bounds from Theorem 5.1. For the first term of \eqref{eq:lapF},
$$ \left\vert tg'^{k\bar{l}}\left(g'^{i\bar{j}}\right)_{k\bar{l}} f_i f_{\bar{j}} \right\vert \leq C_1t |\partial f|^2.$$
Let $\varepsilon > 0$. We bound the second and third terms of \eqref{eq:lapF} with the inequalities
$$ \left\vert tg'^{k\bar{l}}\left(g'^{i\bar{j}}\right)_{k} f_{i\bar{l}} f_{\bar{j}} \right\vert \leq \frac{t}{\varepsilon} |\partial f|^2 + t\varepsilon |\partial \bar{\partial} f |^2 $$
and
$$ \left\vert tg'^{k\bar{l}}\left(g'^{i\bar{j}}\right)_{k} f_i f_{\bar{j}\bar{l}} \right\vert \leq \frac{t}{\varepsilon} |\partial f|^2 + t\varepsilon |D^2 f |^2. $$ where $$|\partial \bar{\partial} f|^2 = g'^{k\bar{l}} g'^{i\bar{j}} f_{i\bar{l}} f_{\bar{j} k}, \ \ |D^2 f|^2 = g'^{k\bar{l}} g'^{i\bar{j}} f_{ik} f_{\bar{j} \bar{l}}.$$
Term six is equal to $t|D^2 f|^2$ and term eight equals $t|\partial \bar{\partial} f|^2$. The fifth and ninth terms of \eqref{eq:lapF} combine to give
\begin{eqnarray} \nonumber
t g'^{k\bar{l}} g'^{i\bar{j}} f_{ik\bar{l}}f_{\bar{j}} + t g'^{k\bar{l}} g'^{i\bar{j}} f_i f_{\bar{j}k\bar{l}} 
& = & 2t \operatorname{Re} \left< \partial f, \partial (g'^{k\bar{l}}f_{k\bar{l}}) \right> - tg'^{i\bar{j}} \left(g'^{k\bar{l}}\right)_i f_{k\bar{l}} f_{\bar{j}} \\ \nonumber
& & \ \ - t g'^{i\bar{j}} \left(g'^{k\bar{l}}\right)_{\bar{j}} f_i f_{k\bar{l}} \\ \nonumber
& \geq & 2t \operatorname{Re} \left< \partial f, \partial (g'^{k\bar{l}}f_{k\bar{l}}) \right> - \frac{t}{\varepsilon} |\partial f|^2 - t\varepsilon |\partial \bar{\partial} f|^2 \nonumber
\end{eqnarray}
We use the definition of $F$ to show
\begin{eqnarray} \nonumber
t g'^{k\bar{l}} g'^{i\bar{j}} f_{ik\bar{l}}f_{\bar{j}} & + & t g'^{k\bar{l}} g'^{i\bar{j}} f_i f_{\bar{j}k\bar{l}} \\ \label{eq:ly1}
& \geq & -2\operatorname{Re} \left< \partial f, \partial F \right> -2t(\alpha - 1)\operatorname{Re} \left<\partial f,\partial\left(f_t\right) \right> - \frac{t}{\varepsilon} |\partial f|^2 - t\varepsilon |\partial \bar{\partial} f|^2.
\end{eqnarray}
Applying equation \eqref{eq:timederiv2} to \eqref{eq:ly1} gives
\begin{eqnarray} \nonumber
t g'^{k\bar{l}} g'^{i\bar{j}} f_{ik\bar{l}}f_{\bar{j}} & + & t g'^{k\bar{l}} g'^{i\bar{j}} f_i f_{\bar{j}k\bar{l}} \\ \nonumber
& \geq & -2\operatorname{Re} \left< \partial f, \partial F \right> -(\alpha - 1)F_t + (\alpha - 1)\left(|\partial f|^2 - \alpha f_t \right) \\ \nonumber
& & \ \ + t(\alpha -1)\left(\frac{\partial}{\partial t}g'^{i\bar{j}}\right)f_i f_{\bar{j}} - t\alpha (\alpha -1)f_{tt} - \frac{t}{\varepsilon} |\partial f|^2 - t\varepsilon |\partial \bar{\partial} f|^2 \\ \nonumber
& \geq & -2\operatorname{Re} \left< \partial f, \partial F \right> -(\alpha - 1)F_t + (\alpha - 1)\left(|\partial f|^2 - \alpha f_t \right) \\ \nonumber
& & \ \ - C_2 t |\partial f|^2 - t\alpha (\alpha -1)f_{tt} - \frac{t}{\varepsilon} |\partial f|^2 - t\varepsilon |\partial \bar{\partial} f|^2. \\ \nonumber
\end{eqnarray}
The final term of \eqref{eq:lapF} becomes, using \eqref{eq:timederiv}
\begin{eqnarray} \nonumber
-\alpha t g'^{k\bar{l}} f_{tk\bar{l}} & = & \alpha t \left(\frac{\partial}{\partial t}g'^{k\bar{l}}\right)f_{k\bar{l}} - \alpha t \frac{\partial}{\partial t} \left(g'^{k\bar{l}} f_{k\bar{l}} \right) \\ \nonumber
& \geq & -\frac{Ct}{\varepsilon} - t\varepsilon |\partial \bar{\partial} f|^2 -\frac{\alpha}{t} F + \alpha F_t + t\alpha (\alpha - 1)f_{tt}. \nonumber
\end{eqnarray}
We put all of the above in to \eqref{eq:lapF}, which shows that
\begin{eqnarray} \nonumber
g'^{k\bar{l}}F_{k\bar{l}} & \geq & F_t - 2\operatorname{Re}\left< \partial f, \partial F\right> - \left(|\partial f|^2-\alpha f_t\right) + t(1-4\varepsilon)|\partial \bar{\partial} f|^2 \\ \nonumber 
& & \ \  + t(1-2\varepsilon)|D^2 f|^2 - t\left(C_1+C_2+\frac{6}{\varepsilon}\right) |\partial f|^2 - \frac{Ct}{\varepsilon}. \nonumber
\end{eqnarray}
Taking $\varepsilon$ sufficiently small and applying the arithmetic-geometric mean inequality $$|\partial \bar{\partial} f|^2 \geq \frac{1}{n}\left(g'^{k\bar{l}} f_{k\bar{l}} \right)^2 = \frac{1}{n}\left(|\partial f|^2 - f_t\right)^2,$$ we see that
$$g'^{k\bar{l}} F_{k\bar{l}} - F_t \geq \frac{t}{2n}\left( |\partial f|^2 - f_t \right)^2 - 2\operatorname{Re} \left< \partial f, \partial F \right> - \left(|\partial f|^2 - \alpha f_t\right) - C t|\partial f|^2 - C t.$$

\end{proof}

Using the previous lemma, we derive an estimate which will be used to prove the Harnack inequality.

\begin{lemma}
There exist constants $C_1$ and $C_2$ depending only on the bounds of the metric $g'$ such that for $t > 0$,  $$|\partial f|^2 - \alpha f_t \leq C_1 + \frac{C_2}{t}.$$ 
\end{lemma}

\begin{proof}

Fix $T > 0$ and let $(x_0,t_0)$ in $M \times [0,T]$ be where $F$ attains its maximum. Note that we can take $t_0 > 0$. Then at $(x_0,t_0)$, from the previous lemma,
\begin{equation}\label{eq:conv1}
\frac{t_0}{2n}\left( |\partial f|^2 - f_t \right)^2 - \left(|\partial f|^2 - \alpha f_t\right) \leq C_1 t_0 |\partial f|^2 + C_2 t_0.
\end{equation}
First we assume that $f_t(x_0,t_0) \geq 0$, then the $\alpha$ in the above inequality can be dropped to give
\begin{equation}
\frac{t_0}{2n}\left( |\partial f|^2 - f_t \right)^2 - \left(|\partial f|^2 - f_t\right) \leq C_1 t_0 |\partial f|^2 + C_2 t_0 . \nonumber
\end{equation}
We factor the above to get
\begin{equation}
\frac{1}{2n}\left(|\partial f|^2 - f_t\right)\left(|\partial f|^2 - f_t - \frac{2n}{t_0}\right) \leq C_1 |\partial f|^2 + C_2. \nonumber
\end{equation}
Hence,
\begin{equation}
|\partial f|^2 - f_t \leq C_3|\partial f| + C_4 + \frac{C_5}{t_0}. \nonumber
\end{equation}
There exists a constant $C_6$ such that 
\begin{equation}
C_3 |\partial f| \leq \left(1-\frac{1}{\alpha}\right)|\partial f|^2 + C_6. \nonumber
\end{equation}
We plug this in to the previous inequality, showing that
\begin{equation}\label{eq:conv4}
\frac{1}{\alpha}|\partial f|^2 - f_t \leq C_7 + \frac{C_5}{t_0}.
\end{equation}
At the point $(x_0,t_0)$, we have
\begin{equation}
F(x_0,t_0) = t_0\left( |\partial f|^2(x_0,t_0) - \alpha f_t(x_0,t_0) \right) \leq C_{8} t_0 + C_5. \nonumber
\end{equation}
Hence for all $x$ in $M$,
\begin{eqnarray} \nonumber
F(x,T) & \leq & F(x_0,t_0) \\ \nonumber
& \leq & C_{8} t_0 + C_5 \\ \nonumber
& \leq & C_{8} T + C_5 \nonumber
\end{eqnarray}
completing the proof for this case.

Now we consider the case where $f_t(x_0,t_0) < 0$. Using \eqref{eq:conv1} at the point $(x_0,t_0)$,
\begin{equation}
\frac{t_0}{2n}|\partial f|^4 - |\partial f|^2 \leq C_1 t_0 |\partial f|^2 + C_2 t_0 - \alpha f_t. \nonumber
\end{equation}
We factor the above to get
\begin{equation}
|\partial f|^2 \left( \frac{1}{2n} |\partial f|^2 - \frac{1}{t_0} - C_1\right) \leq C_2 - \frac{\alpha}{t_0} f_t. \nonumber
\end{equation}
Hence,
\begin{equation}\label{eq:conv2}
|\partial f|^2 \leq C_{3} + \frac{C_{4}}{t_0} - \frac{1}{2}f_t.
\end{equation}
We use \eqref{eq:conv1} again and the condition that $f_t(x_0,t_0) < 0$ to see that
\begin{equation}
\frac{t_0}{2n}f_t^2+\alpha f_t \leq C_1 t_0|\partial f|^2 + |\partial f|^2 + C_2 t_0. \nonumber
\end{equation}
By factoring the above, we show that
\begin{equation}
\frac{1}{2n}\left(-f_t\right)\left(-f_t-\frac{2n\alpha}{t_0}\right) \leq C_1|\partial f|^2 + \frac{1}{t_0} |\partial f|^2 + C_2. \nonumber
\end{equation}
And so
\begin{equation}\label{eq:conv3}
-f_t \leq C_{5} + \frac{C_{6}}{t_0} + \frac{1}{2} |\partial f|^2.
\end{equation}
We plug \eqref{eq:conv3} in to \eqref{eq:conv2}, arriving at
\begin{equation}
|\partial f|^2 \leq C_3 + \frac{C_4}{t_0} + \frac{C_5}{2} + \frac{C_6}{2t_0} + \frac{1}{4}|\partial f|^2. \nonumber
\end{equation}
This provides the following estimate for $|\partial f|^2$:
\begin{equation}\label{eq:thing1}
|\partial f|^2 \leq C_7 + \frac{C_8}{t_0}.
\end{equation}
Similarly, we can show that
\begin{equation}\label{eq:thing2}
-\alpha f_t \leq C_9 + \frac{C_{10}}{t_0}.
\end{equation}
We add \eqref{eq:thing1} and \eqref{eq:thing2} to obtain the estimate
\begin{equation}
|\partial f|^2 - \alpha f_t \leq C_{11} + \frac{C_{12}}{t_0}. \nonumber
\end{equation}
Repeating the argument after \eqref{eq:conv4} completes this case and hence the proof. 

\end{proof}

We use the previous lemma to derive a Harnack inequality similar to that of Li and Yau in the case of a Hermitian manifold.

\begin{lemma}\label{lemma:harnack2}
For $0 < t_1 < t_2$, $$\sup_{x\in M} u(x,t_1) \leq \inf_{x\in M} u(x,t_2) \left(\frac{t_2}{t_1}\right)^{C_2} \exp{\left( \frac{C_3}{t_2-t_1} + C_1(t_2-t_1) \right)}$$ where $C_1, C_2$ and $C_3$ are constants depending only on the bounds of the metric $g'$.
\end{lemma}

\begin{proof}
Let $x,y \in M$, and define $\gamma$ to be the minimal geodesic (with respect to the initial metric $g_{i\bar{j}}$) with $\gamma(0) = y$ and $\gamma(1)=x$. Define a path $\zeta : [0,1] \to M \times [t_1,t_2]$ by $\zeta(s)=\left(\gamma(s),(1-s) t_2 + s t_1\right)$. Then using Lemma 6.2,

\begin{eqnarray}
\log{\frac{u(x,t_1)}{u(y,t_2)}} & = & \int_0^1 \frac{d}{ds} f(\zeta(s)) \ ds \\ \nonumber
& = & \int_0^1 \left( \left< \dot{\gamma},2\partial f \right> - (t_2-t_1)f_t\right) \ ds \\ \nonumber
& \leq & \int_0^1 -\frac{t_2-t_1}{\alpha}\left(|\partial f| - \frac{\alpha |\dot{\gamma}|}{(t_2-t_1)}\right)^2 + \frac{\alpha|\dot{\gamma}|^2}{(t_2-t_1)} \\ \nonumber
& & \ \ \ + C_{1}(t_2-t_1) + C_{2}\frac{t_2-t_1}{t} \ ds \\ \nonumber
& \leq & \int_0^1 \frac{C_{19}}{t_2-t_1} + C_{17}(t_2-t_1) + C_{18} \frac{t_2-t_1}{t} \ ds \\ \nonumber
& = & \frac{C_{3}}{t_2-t_1} + C_{1}(t_2-t_1) + C_{2} \log{\left(\frac{t_2}{t_1}\right)}
\end{eqnarray} 
Exponentiating both sides completes the proof.
\end{proof}

\section{Convergence of the flow}

With the Harnack inequality, we complete the proof of the main theorem by showing the convergence of $\tphi$ (cf. \cite{Cao}). 

\begin{proof}

Define $u = \frac{\partial \varphi}{\partial t}$. Then $$\frac{\partial u}{\partial t} = g'^{i\bar{j}} \partial_i \partial_{\bar{j}} u.$$
Let $m$ be a positive integer and define $$\xi_m(x,t) = \sup_{y\in M} u(y,m-1) - u(x,m-1+t)$$ $$\psi_m(x,t) = u(x,m-1+t) - \inf_{y\in M} u(y,m-1).$$ These functions satisfy the heat type equations $$\frac{\partial \xi_m}{\partial t} = g'^{i\bar{j}}(m-1+t) \partial_i \partial_{\bar{j}} \xi_m$$ $$\frac{\partial \psi_m}{\partial t} = g'^{i\bar{j}}(m-1+t) \partial_i \partial_{\bar{j}} \psi_m.$$ 

First consider the case where $u(x,m-1)$ is not constant. Then $\xi_m$ is positive for some $x$ in $M$ at time $t = 0$. By the maximum principle, $\xi_m$ must be positive for all $x$ in $M$ when $t > 0$. Similarly, $\psi_m$ is positive everywhere when $t > 0$. Hence we can apply Lemma \eqref{lemma:harnack2} with $t_1 = \frac{1}{2}$ and $t_2 = 1$ to get 
$$\sup_{x\in M} u(x,m-1) - \inf_{x\in M} u\left(x,m-\frac{1}{2}\right) \leq C \left( \sup_{x\in M} u(x,m-1) - \sup_{x\in M} u(x,m)\right)$$ $$\sup_{x\in M} u\left(x,m- \frac{1}{2} \right) - \inf_{x\in M} u\left(x,m-1 \right) \leq C \left( \inf_{x\in M} u(x,m) - \inf_{x\in M} u(x,m-1)\right).$$
We define the oscillation $\theta(t) = \sup_{x\in M} u(x,t) - \inf_{x\in M} u(x,t)$. Adding the above inequalities gives $$\theta(m-1) + \theta \left(m-\frac{1}{2}\right) \leq C\left( \theta(m-1) - \theta(m) \right).$$
Rearranging and setting $\delta = \frac{C-1}{C} < 1$ yields $$\theta(m) \leq \delta \theta(m-1).$$ By induction, $$\theta(t) \leq Ce^{-\eta t}$$
where $\eta = -\log \delta$. Note that if $u(x,m-1)$ is constant, this inequality is still true. 

Fix $(x,t)$ in $M\times [0,\infty)$. Since $$\int_M \frac{\partial \tilde{\varphi}}{\partial t} \omega^n  = 0,$$ there exists a point $y$ in $M$ such that $\frac{\partial \tilde{\varphi}}{\partial t}(y,t) = 0$. 

\begin{eqnarray}
\left\vert \frac{\partial \tilde{\varphi}}{\partial t}(x,t) \right\vert & = & \left\vert \frac{\partial \tilde{\varphi}}{\partial t}(x,t) - \frac{\partial \tilde{\varphi}}{\partial t}(y,t) \right\vert \\ \nonumber
& = & \left\vert \frac{\partial \varphi}{\partial t}(x,t) - \frac{\partial \varphi}{\partial t}(y,t) \right\vert \\ \nonumber
& \leq & Ce^{-\eta t}.
\end{eqnarray} 
 
Consider the quantity $Q_2 = \tilde{\varphi} + \frac{C}{\eta}e^{-\eta t}$. Then by construction, $$\frac{\partial Q_2}{\partial t} \leq 0.$$ Since $Q_2$ is bounded and monotonically decreasing, it tends to a limit as $t \to \infty$, call it $\tilde{\varphi}_{\infty}$. But $$\lim_{t\to \infty} \tilde{\varphi} = \lim_{t\to \infty} Q_2 - \lim_{t\to \infty} \frac{C}{\eta}e^{-\eta t} = \tilde{\varphi}_{\infty}.$$

To show that the convergence of $\tphi$ to $\tilde{\varphi}_{\infty}$ is actually $C^{\infty}$, suppose not. Then there exists a time sequence $t_m \to \infty$ such that for some $\varepsilon > 0$ and some integer $k$,  
\begin{equation}\label{eq:convcinf}
||\tphi(x,t_m) - \tphi_\infty||_{C^k} > \varepsilon, \ \ \forall m.
\end{equation} 
However, since $\tphi$ is bounded in $C^{\infty}$ there exists a subsequence $t_{m_j} \to \infty$ such that $\tphi(x,t_{m_j}) \to \tphi'_{\infty}$ as $j \to \infty$ for some smooth function $\tphi'_{\infty}$. By \eqref{eq:convcinf}, $\tphi'_{\infty} \neq \tphi_{\infty}$. This is a contradiction, since $\tphi \to \tphi_{\infty}$ pointwise. Hence the convergence of $\tphi$ to $\tphi_{\infty}$ is $C^{\infty}$.
 
We observe that $\tphi$ solves the parabolic flow $$\frac{\partial \tphi}{\partial t} = \log \frac{\det(g_{i\bar{j}} + \partial_i \partial_{\bar{j}} \tphi)}{\det g_{i\bar{j}}} - F - \int_M \frac{\partial \varphi}{\partial t} \ \omega^n.$$
Taking $t$ to infinity, we see that $\tphi_\infty$ solves the elliptic Monge-Amp\`{e}re equation $$\log \frac{\det (g_{i\bar{j}} + \partial_i \partial_{\bar{j}} \tphi_\infty)}{\det g_{i\bar{j}}} = F + b$$ where $$b = \int_M \left( \log \frac{\det (g_{i\bar{j}} + \partial_i \partial_{\bar{j}} \tphi_\infty)}{\det g_{i\bar{j}}} - F \right) \ \omega^n.$$ This combined with Theorem 5.1 completes the proof of the main theorem, and also provides a parabolic proof of the main theorem in \cite{TW2}.
\end{proof}

\bigskip
\noindent
{\bf Acknowledgements}

The author would like to thank his thesis advisor Ben Weinkove for countless helpful discussions and advice. The author would also like to thank Valentino Tosatti for helpful suggestions. The author thanks the referee for a careful reading of the first version of this paper and for making a number of helpful suggestions and comments.

The contents of this paper will appear in the author's forthcoming PhD thesis.

\bigskip
\noindent
Mathematics Department, University of California, San Diego, 9500 Gilman Drive \#0112, La Jolla CA 92093

\end{document}